\documentclass[11pt]{amsart}%
\usepackage{amsmath,amssymb,amsthm,amscd, amsxtra, amsfonts}
\usepackage{amsmath}
\usepackage{amsfonts}
\usepackage{amssymb}
\usepackage{graphicx}%
\newcommand{\N}{{\Bbb Z} _{> 0} }
\newcommand{\Zp} {\Z _ {\ge 0} }

\newcommand{\Z}{\Bbb Z}

\providecommand{\U}[1]{\protect\rule{.1in}{.1in}}

\newcommand{\HVir}{\mathcal H}
\newcommand{\Y}{\mathcal Y}

\newtheorem{theorem}{Theorem}[section]
\newtheorem{corollary}[theorem]{Corollary}
\newtheorem{lemma}[theorem]{Lemma}

\newtheorem{remark}[theorem]{Remark}
\newtheorem{example}[theorem]{Example}

\newtheorem{proposition}[theorem]{Proposition}
\def \a{\alpha }
\def \b {\beta}

\newcommand{\bea}{\begin{eqnarray}}
\newcommand{\eea}{\end{eqnarray}}

\begin{document}

\keywords{twisted Heisenberg-Virasoro algebra, $W(2,2)$-algebra, Free fields, Schur polynomials, Singular vectors}
\title[{Free field realization of the twisted Heisenberg-Virasoro algebra at level zero and its applications} ]{Free field realization of the twisted Heisenberg-Virasoro algebra at level zero and its applications}

  \subjclass[2000]{
Primary 17B69, Secondary 17B67, 17B68, 81R10}

\author{Dra\v zen Adamovi\' c }
\curraddr{Department of Mathematics, University of Zagreb, Bijeni\v cka 30, 10 000
Zagreb, Croatia}
\email{adamovic@math.hr}
\author{Gordan Radobolja }
\curraddr{Faculty of Science, University of Split, Teslina 12,
21 000 Split, Croatia }
\email{gordan@pmfst.hr}
\date{}
\maketitle

\begin{abstract}
We investigate the free field realization of the twisted
Hei\-sen\-berg-Virasoro algebra $\HVir$ at level zero. We completely
describe the structure of the associated Fock representations. Using
vertex-algebraic methods and screening operators we construct singular vectors
in certain Verma modules as Schur polynomials. We completely solve the irreducibility problem for the tensor products of irreducible highest weight modules with intermediate series. We also determine the fusion rules for an interesting subcategory of $\HVir$--modules. Finally, as an application we present a
free-field realization of the $W(2,2)$-algebra and interpret the $W(2,2)$--singular vectors as $\HVir$--singular vectors in Verma modules.
\end{abstract}

\baselineskip=14pt
\newenvironment{demo}[1]{\vskip-\lastskip\medskip\noindent{\em#1.}\enspace
}{\qed\par\medskip}

\markboth{Dra\v zen Adamovi\' c and Gordan Radobolja} { } \pagestyle{myheadings}

\section{Introduction}

In this paper we study the representation theory  of the twisted Heisenberg-Virasoro algebra at level zero by   using methods from the theory of vertex algebras. An emphasis will be put on the construction and explicit realization of modules, intertwining operators and the determination of fusion rules. The highest weight representation theory of the twisted  Heisenberg-Virasoro algebra at level zero was developed by Y. Billig (cf. \cite{Billig}, \cite{Billig2}). This Lie algebra is important because it appears in the representation theory of toroidal Lie algebras. We shall show in this paper that it is also interesting from the vertex-algebraic point of view. In the present paper, instead of applying the highest weight representation theory developed by Y. Billig, we shall apply the concept of a free field realization of modules.

Recall that the twisted Heisenberg-Virasoro algebra is the infinite-dimen\-si\-o\-nal complex Lie
algebra $\HVir$ with basis
$$
\{L(n),I(n):n\in\mathbb{Z}\}\cup\{C_{L},C_{LI},C_{I}\}
$$
and commutation relations:
\bea
&& \left[  L(n),L(m)\right]  =(n-m)L(n+m)+\delta_{n,-m}\frac{n^{3}-n}{12}%
C_{L},\\
&& \left[  L(n),I(m)\right]  =-mI(n+m)-\delta_{n,-m}( n^{2}+n)
C_{LI},\\
 && \left[  I(n),I(m)\right]  =n\delta_{n,-m}C_{I},\\ && \left[  {\HVir} ,C_{L}\right]  =\left[  {\HVir},C_{LI}\right]  =\left[
\HVir,C_{I}\right]  =0.
\eea

  Using the results and concepts of \cite{Ar} we see that when the central element $C_I$  of the  Heisenberg subalgebra of $\HVir$  acts non-trivially, then every highest weight $\HVir$--module is a tensor product of a highest weight module for the Virasoro algebra and an irreducible module for the Heisenberg algebra. So in this case, a free-field realization of $\HVir$--modules can be obtained by using the  free-field realization of the Virasoro algebra, which is well-known (cf. \cite{FF}). Therefore in this paper we shall be  focused only  on the case of level zero, i.e., when  $C_I$ acts trivially.

  We should mention that a free field realization of certain indecomposable modules at level zero was constructed and analyzed in \cite{JJ}.

Let $V^{\HVir}(c_{L},c_{I},c_{L,I},h,h_{I})$ denote the Verma module with  highest weight $(c_{L},c_{I},c_{L,I},h,h_{I})$ (cf. \cite{Billig}).

When $c_I=0$, we will write $V^{\HVir}$ for $V^{\HVir}(c_L,0,c_{L,I},h,h_I)$.

\begin{theorem}
\cite{Billig} Assume that $c_{I}=0$ and $c_{L,I}\neq0$.

\begin{description}
\item[(i)] If $\frac{h_{I}}{c_{LI}}\notin\mathbb{Z}$ or $\frac{h_{I}}{c_{L,I}%
}=1$, then the Verma module $V^{\HVir}$ is irreducible.

\item[(ii)] If $\frac{h_{I}}{c_{LI}}\in\mathbb{Z\setminus\{}1\}$, then
$V^{\HVir}$ has a singular vector $v\in V_{p}$, where
$p=|\frac{h_{I}}{c_{LI}}-1|$. The quotient module $L^{\HVir}(c_{L},0,c_{L,I}%
,h,h_{I})=V^{\HVir}/U(\HVir)v$ is irreducible and its character is
$$
\operatorname*{char}L^{\HVir}(c_{L},0,c_{L,I},h,h_{I})=q^{h}(1-q^{p})\prod_{j\geq
1}(1-q^{j})^{-2}.
$$
\end{description}
\end{theorem}

In this paper we shall first present an explicit formula for  singular
vectors in the case $h_{I}/c_{L,I}-1=p\geq1$, and then a partial formula in the case $1-h_{I}/c_{L,I}=p\geq1$.

Recall that the Schur polynomials $S_{r}(x_{1},x_{2},\cdots)$ in variables
$x_{1},x_{2},\cdots$ are defined by the following equation:
$$
\exp\left(  \sum_{n=1}^{\infty}\frac{x_{n}}{n}y^{n}\right)  =\sum
_{r=0}^{\infty}S_{r}(x_{1},x_{2},\cdots)y^{r}.
$$

\begin{theorem}
Assume that $c_{L,I}\neq0$ and $p=\frac{h_{I}}{c_{L,I}}-1\in{\mathbb{Z}}_{>0}$.
Then
\begin{equation}
S_{p}(-\frac{I(-1)}{c_{L,I}},-\frac{I(-2)}{c_{L,I}},\cdots)v_{h,h_{I}}
\label{sing-hvir1}%
\end{equation}
is a singular vector of conformal weight $p$ in the Verma module $V^{\HVir}$.
\end{theorem}

The proof of this theorem will use a free field realization of the
Heisenberg-Virasoro vertex algebra, the screening operator
$$
Q=\mbox{Res}_{z}Y(e^{\frac{\alpha+\beta}{-c_{L,I}}},z)
$$
and the vertex-operator formulas in lattice vertex algebras (cf. Section
\ref{free}).

Unfortunately this approach does not provide an explicit realization of the
singular vector in the case $h_{I}/c_{L,I}\in-{\mathbb{Z}}_{\geq0}$. But in this case we will use the free-field realization to present a cohomological
realization of the  irreducible highest weight modules. We will show that
the irreducible highest weight module $L^{\HVir}(c_{L},0,c_{L,I},h,h_{I})$ can be realized as a submodule of a suitably chosen Fock space $\mathcal{F}_{r,s}$ such that
$$
L^{\HVir} (c_{L},0,c_{L,I},h,h_{I})=\mbox{Ker}_{\mathcal{F}_{r,s}}Q.
$$

We apply our construction to the following two (closely related) problems in the representation theory of the twisted Heisenberg-Virasoro algebra $\HVir$:
\begin{itemize}
\item[(A)] The irreducibility of tensor products: $V'_{\alpha, \beta, F} \otimes L^{\HVir} (c_{L},0,c_{L,I},h,h_{I})$, where
$V'_{\alpha, \beta, F}$ is an irreducible $\HVir$--module from the intermediate series.
\item[(B)] The determination of fusion rules for irreducible $\HVir$--modules from the vertex-algebraic point of view.
\end{itemize}
Problem (A) was posed in \cite{R1} (see also \cite{LZ}) but for the complete solution one needs explicit formulas for singular vectors. These formulas are obtained in the present paper, so in Theorems \ref{-+} and \ref{red-case2} we present the complete solution to  problem (A).

For  problem (B), in Theorem \ref{classif} we determine the fusion rules for simple highest weight modules
$$   L^{\HVir}(c_L, 0, c_{L,I}, h, h_I,  ) \quad \mbox{
such that} \quad  h_I /c_{L,I} -1 \in {\Z} \setminus \{ 0 \}.$$
It is important to notice that all non-trivial intertwining operators from  the fusion rules formula are explicitly constructed  using our free field realization.

Our approach can be also used for the determination of singular vectors in  the
Lie algebra $W(2,2)$ (cf. Theorem \ref{singular-w22}). We present  a free-field realization of $W(2,2)$  which enables us to  show that the Verma module $V^{\HVir}$ has the structure of a Verma module over the $W(2,2)$--algebra. Therefore, the singular vector (\ref{sing-hvir1}) becomes a singular vector over the $W(2,2)$--algebra.

\section{ Free-field realization of the  Heisenberg-Virasoro vertex algebra
and screening operators}
\label{free}

In this section we shall present a free field realization of the  Heisenberg-Virasoro vertex algebra $L^{\HVir%
}(c_{L},c_{L,I})$ at level zero. We shall realize $L^{\HVir%
}(c_{L},c_{L,I})$ as a subalgebra of the Heisenberg vertex algebra $M(1)$ generated by two Heisenberg fields. This will imply that every $M(1)$--module $M(1, \gamma)$ is a module for the vertex algebra $L^{\HVir%
}(c_{L},c_{L,I})$ and therefore for the twisted Heisenberg-Virasoro algebra $\HVir$. We completely describe the structure of $M(1, \gamma)$ as a $\HVir$--module. We will also construct screening operators which will be very useful for the determination of singular vectors.

A free--field realization of certain $\HVir$--modules  also appeared in \cite{JJ}. The main difference with our approach is in the fact  that we use screening operators to construct singular and cosingular vectors.

\vskip 5mm

Define the following hyperbolic lattice $L= {\mathbb{Z}} {\alpha} +
{\mathbb{Z}} {\beta}$ such that
$$
\langle\alpha, \alpha\rangle= - \langle\beta, \beta\rangle= 1, \ \langle
\alpha, \beta\rangle= 0.
$$
Let ${\frak h} = {\Bbb C} \otimes_{\Z} L$ and extend the form $\langle \cdot, \cdot \rangle $ on ${\frak h}$.
We can consider ${\frak h}$ as an abelian Lie algebra. Let $\widehat{\frak h}= {\Bbb C}[t,t^{-1}]  \otimes {\frak h} \oplus {\Bbb C} c$ be the affinization of ${\frak h}$. Let $\gamma \in {\frak h}$ and consider $\widehat{\frak h}$--module
$$ M(1, \gamma) := U(\widehat{\frak h}) \otimes _{ U ( {\Bbb C} [t] \otimes {\frak h} \oplus {\Bbb C} c ) } {\Bbb C}  $$
where $ t {\Bbb C}   [t] \otimes {\frak h} $ acts trivially on ${\Bbb C}$, $ {\frak h}$ acts as $ \langle \delta , \gamma \rangle $ for $\delta \in {\frak h} $ and $c$ acts as $1$.
We shall denote the highest weight vector in $M(1, \gamma)$ by $e ^{\gamma}$.

We shall write $M(1)$ for $M(1,0)$.  For $h \in {\frak h}$ and $n \in {\Bbb Z}$ we write $h(n)$ for $t ^n \otimes h$. Set $h(z) = \sum_{n \in {\Bbb Z} } h(n)  z ^{-n-1}$. Then $M(1)$ is a vertex algebra generated by the fields $h(z)$, $h \in {\frak h}$. Moreover, $M(1, \gamma)$ for $\gamma \in {\frak h}$, are irreducible $M(1)$--modules.

 Let $V_{L} = M(1) \otimes{\mathbb{C}}[L] $  be the vertex algebra associated to the lattice $L$, where ${\mathbb{C}}[L] $ is the group algebra of $L$.
Details about lattice vertex algebras can be found in \cite{K}, \cite{LL}.
The lattice vertex algebra $V_L$ appeared in \cite{Billig2} in the study
of toroidal vertex algebras.

Define the Heisenberg and Virasoro vectors
\bea &&  I  = \alpha(-1) + \beta (-1) \label{Heisenberg} \\
&& \omega =\frac{1}{2} \alpha(-1) ^2 - \frac{1}{2} \beta (-1) ^2 + \lambda \alpha(-2)  + \mu \beta (-2).  \label{Virasoro} \eea
Then the components of the fields
$$
I(z)=Y(I,z)=\sum_{n\in{\mathbb{Z}}}I(n)z^{-n-1},\quad L(z)=Y(\omega
,z)=\sum_{n\in{\mathbb{Z}}}L(n)z^{-n-2}%
$$
satisfy the commutation relations for the twisted Heisenberg-Virasoro Lie
algebra $\HVir$ such that
\begin{equation}
c_{L}=2-12(\lambda^{2}-\mu^{2}),\ c_{L,I}=\lambda-\mu. \label{par-1}%
\end{equation}
From (\ref{par-1}) it follows that
$$ \lambda = \frac{2-c_L}{24 c_{L,I} } + \frac{1}{2} c_{L,I}, \ \mu =  \frac{2-c_L}{24 c_{L,I} } - \frac{1}{2} c_{L,I}. $$
In particular, $I$ and $\omega$ generate the simple Heisenberg-Virasoro vertex
algebra $L^{\HVir} (c_{L},0,c_{L,I},0,0)$ which we shall denote by $L^{\HVir%
}(c_{L},c_{L,I})$.

Note that for every $r,s\in{\mathbb{C}}$  $e^{r\alpha+s\beta}$ is a singular
vector for the twisted Heisenberg-Virasoro Lie algebra $\HVir$ in $M(1, r\alpha+s\beta)$   and
$U(\HVir).e^{r\alpha+s\beta}$ is a highest weight module with the highest
weight $(h,h_{I})$ where
\begin{equation}
h=\Delta_{r,s}=\frac{1}{2}r^{2}-\frac{1}{2}s^{2}-\lambda r+\mu s,\ h_{I}=r-s.
\label{par-2}%
\end{equation}

One can easily see that for every $(h,h_{I})\in{\mathbb{C}}^{2}$, $h_I \ne c_{L,I}$ there exists a unique $(r,s)\in{\mathbb{C}}^{2}$ such that (\ref{par-2}) holds. More precisely, we have:

\begin{proposition} \label{self-dual}
\item[(1)] Let $(h, h_I) \in {\Bbb C} ^2$,  $h_I \ne c_{L,I}$. Then there exist unique $r, s \in {\Bbb C}$ such that $e ^{r \alpha + s \beta}$ is a highest weight vector of the highest weight $(h, h_I)$.
\item[(2)] For every $r,s \in {\Bbb C}$ such that $r-s = \lambda-\mu = c_{L,I}$ $e^{r\alpha +s \beta}$ is a highest weight vector of weight
$$ (h, h_I)= (\frac{c_L-2}{24}, c_{L,I}). $$
    \end{proposition}

\begin{remark}
It is very interesting that our free-field realization in the case $h_I=c_{L,I}$ gives only a realization of highest weight module with $h=(c_L -2) / 24$.
\end{remark}

 In particular, if we take $r=s=-\frac{1}{c_{L,I}}$ we get that $u=e^{-\frac
{\alpha+\beta}{c_{L,I}}}$ is a highest weight vector of highest weight $(1,0)$.

Let
$$
Q= \mbox{Res}_{z} Y(u,z) = u_{0}.
$$

The vertex operator $Y(u,z)$ acts in the following sense.

Let $\phi=-\frac{\alpha+\beta}{c_{L,I}}$ and $D$ be the lattice $D={\mathbb{Z}%
}\phi$. Then one can show that
$$
\overline{V}_{D}=M(1)\otimes{\mathbb{C}}[D]
$$
is a vertex algebra (cf. \cite{BDT}). The vertex
operator $Y(u,z)$ and its component $Q$ are well-defined on every
$\overline{V}_{D}$--module. Let $D^{0}$ denote the dual lattice of $D$. In our case
$$
\lambda\in D^{0}\quad\iff\quad\langle\lambda,\alpha+\beta\rangle\in
{c_{L,I}\mathbb{Z}}.
$$
Then for every $\lambda\in D^{0}$
$$
\overline{V}_{\lambda+D}=M(1)\otimes{\mathbb{C}}[\lambda+D]
$$
is a $\overline{V}_{D}$--module (cf. \cite{BDT},  \cite{LL}, \cite{K} ).

\begin{remark}
The representation theory of the vertex algebra $\overline{V}_{D}$ was developed in \cite{BDT} and \cite{LW}. This vertex algebra also appeared  in \cite{Billig2}, \cite{BF} where it was denoted by $V_{Hyp}^+$.
\end{remark}

\begin{lemma}

$$
\lbrack Q,L(n)]=[Q,I(n)]=0\quad\forall n\in{\mathbb{Z}}.
$$
Therefore, $Q$ is a screening operator which commutes with the action of
$\HVir$.
\end{lemma}

\begin{proof}
Since
$$
L(0)u=u,\ L(n)u=0\quad\mbox{for}\ n\geq1
$$
using commutator formulas for vertex operators  we get
$$
\lbrack L(n),u_{m}]=-mu_{n+m}%
$$
which implies that $Q$ commutes with the action of the Virasoro algebra. Since
$$
I(n)u=0\quad\forall n\geq0,
$$
the commutator formula implies that the components of the fields $I(z)$ and $Y(u,z)$ commute. The proof follows.
\end{proof}

Since $Q$ commutes with the action of the twisted Heisenberg-Virasoro Lie
algebra, then for every $j \in{\mathbb{Z} _{> 0} }$, $Q ^{j} e^{r \alpha+ s
\beta}$ is either zero or a singular vector. Moreover, $Q$ is preserving the gradation, so if $Q ^{j} e^{r \alpha+ s
\beta} \ne 0$, then $U(\HVir). e^{r \alpha+ s
\beta}$ and $U(\HVir). Q ^j e^{r \alpha+ s
\beta}$ are highest weight $\HVir$--modules having same highest weight.

Denote by $\mathcal{F}_{r,s}$ the $M(1)$--module generated by $e^{r\alpha
+s\beta}$. Define $c_{L},c_{L,I},h_{I},h$ as above. We can consider
$\mathcal{F}_{r,s}$ as a module for the Heisenberg-Virasoro vertex algebra
$L^{\HVir}(c_{L},c_{L,I})$ (and therefore for the twisted
Heisenberg-Virasoro Lie algebra $\HVir$). Clearly, $U(\HVir
).e^{r\alpha+s\beta}$ is a highest weight $\HVir$--module. Let $v_{h,h_{I}}$ be the highest-weight vector in the Verma module $V^{\HVir}(c_{L},0,c_{L,I},h,h_{I})$. There is a surjective $\HVir$--homomorphism
$$
\Phi:V^{\HVir} (c_{L},0,c_{L,I},h,h_{I})\rightarrow U(\HVir).e^{r\alpha+s\beta
},\ \Phi(v_{h,h_{I}})=e^{r\alpha+s\beta}.
$$
Let
$$
\mathcal{I}={\mathbb{C}}[I(-1),I(-2),\cdots]v_{h,h_{I}}.
$$
By construction we know that $\Phi|\mathcal{I}$ is injective. We have:

\begin{proposition}
\label{free-Verma} Assume that
$$
\frac{h_{I}}{c_{L,I}}-1\notin-{\mathbb{Z}_{>0}}.
$$
Then $\mathcal{F}_{r,s}\cong V^{\HVir} (c_{L},0,c_{L,I},h,h_{I})$ as $L^{\HVir%
}(c_{L},c_{L,I})$--modules.
\end{proposition}

\begin{proof}
First we assume that $\frac{h_{I}}{c_{L,I}}-1\notin{\mathbb{Z}}.$ Then
$V^{\HVir} (c_{L},0,c_{L,I},h,h_{I})$ is an irreducible $\HVir$--module, and
since its $q$--character coincides with the $q$--character of $\mathcal{F}%
_{r,s}$, we get $\mathcal{F}_{r,s}\cong V^{\HVir}(c_{L},0,c_{L,I},h,h_{I})$.

Assume that $\frac{h_{I}}{c_{L,I}}-1=p\in{\mathbb{Z}_{>0}}$. Then the results
of \cite{Billig} easily imply that every singular vector in the Verma
$V^{\HVir} (c_{L},0,c_{L,I},h,h_{I})$ must belong to $\mathcal{I}$. Assume that
$\mbox{Ker}(\Phi)\neq0$. Then the fact that $\Phi|\mathcal{I}$ is injective
implies that $V^{\HVir}(c_{L},0,c_{L,I},h,h_{I})$ contains a singular vector which
does not belong to $\mathcal{I}$. This is in contradiction with \cite{Billig}.
So $\mbox{Ker}(\Phi)=0$ and therefore $V^{\HVir}(c_{L},0,c_{L,I},h,h_{I})\cong
U(\HVir).e^{r\alpha+s\beta}\subseteq\mathcal{F}_{r,s}$. Now the equality
of $q$--characters implies the statement.
\end{proof}

Recall the definition of a contragradient module from \cite{FHL}. Assume
that $V$ is a vertex operator algebra, $(M,Y_{M})$ is a graded $V$--module
such that $M=\oplus_{n=0}^{\infty}M(n)$ and that there is $\gamma
\in{\mathbb{C}}$ such that $L(0)|M(n)\equiv(\gamma+n)\mbox{Id}$, $\dim
M(n)<\infty$. The contragradient module $M^{\ast}$ is defined as follows. For every $n\in{\mathbb{Z}_{>0}}$ let $M(n)^{\ast}$ be the dual vector space and
$M^{\ast}=\oplus_{n=0}^{\infty}M(n)^{\ast}$. Consider the natural pairing
$\langle\cdot,\cdot\rangle:M^{\ast}\otimes M\rightarrow\mathbb{C}$. Define the linear map $Y_{M^{\ast}}:V\rightarrow\mbox{End}M^{\ast}[[z,z^{-1}]$ such that
$$
\langle Y_{M^{\ast}}(v,z)w^{\prime},w\rangle=\langle w^{\prime},Y_{M}%
(e^{zL(1)}(-z^{-2})^{L(0)}v,z^{-1})w\rangle
$$
for each $v\in V$, $w\in M$, $w^{\prime\ast}$. Then $(M^{\ast},Y_{M^{\ast}})$
carries the structure of a $V$--module.

Let us now consider the case $M=\mathcal{F}_{r,s}$. We get
\begin{align}
\langle L(n)w^{\prime},w\rangle &  =\langle w^{\prime},L(-n)w\rangle
\nonumber\\
\langle Qw^{\prime},w\rangle=-  &  \langle w^{\prime},Qw\rangle\nonumber\\
\langle\alpha(n)w^{\prime},w\rangle &  =\langle w^{\prime},(-\alpha
(-n)+2\lambda\delta_{n,0})w\rangle\nonumber\\
\langle\beta(n)w^{\prime},w\rangle &  =\langle w^{\prime},(-\beta
(-n)-2\mu\delta_{n,0})w\rangle\nonumber\\
\langle I(n)w^{\prime},w\rangle &  =\langle w^{\prime},(-I(-n)+2c_{L,I}%
\delta_{n,0})w\rangle.\nonumber
\end{align}

We get the following result:

\begin{lemma}
We have:

\item[(1)]
$$
\mathcal{F}_{r,s} ^{*} \cong\mathcal{F}_{2 \lambda-r, 2 \mu- s }.
$$

\item[(2)]
$$
L^{\HVir} (c_{L},0,c_{L,I},h,h_{I})^{\ast}\cong L^{\HVir} (c_{L},0,c_{L,I},h,-h_{I}+2c_{L,I})
$$

\end{lemma}
Recall that a vector $v$ in a $\HVir$--module $M$ is called cosingular (or subsingular) if there is a submodule $N \subset M$ such that $v+ N$ is a singular vector in the quotient module $M/N$.
\begin{proposition}
\label{struktura}
Assume that
$$
\frac{h_{I}}{c_{L,I}}-1=-p,\quad p\in{\mathbb{Z}_{>0}}.
$$
As a $L^{\HVir}(c_{L},c_{L,I})$--module $\mathcal{F}_{r,s}$ is generated
by $e^{r\alpha+s\beta}$ and a family of cosingular vectors $\{v_{n,p}%
:n\geq1\}$, such that
$$
Q^{n}v_{n,p}=e^{r\alpha+s\beta-n\frac{\alpha+\beta}{c_{L,I}}}.
$$
In particular,
$$
L^{\HVir}(c_{L},0,c_{L,I},h,h_{I})=\mbox{Ker}_{\mathcal{F}_{r,s}}Q.
$$
\end{proposition}

\begin{proof}
First recall the following formula for characters of irreducible modules (cf.
\cite{Billig}):
$$
ch_{q}L^{\HVir}(c_{L},0,c_{L,I},h+np,h_{I})=q^{h+np}(1-q^{p}%
)\prod_{j\geq1}(1-q^{j})^{-2}.
$$
This implies
\begin{equation}
ch_{q}\mathcal{F}_{r,s}=q^{h}\prod_{j\geq1}(1-q^{j})^{-2}=\sum_{n\geq0}%
ch_{q}L^{\HVir}(c_{L},0,c_{L,I},h+np,h_{I}). \label{decomp}%
\end{equation}
Let $\phi=-\frac{\alpha+\beta}{c_{L,I}}$. Consider now the contragradient
module $(\mathcal{F}_{r,s})^{\ast}\cong\mathcal{F}_{2\lambda-r,2\mu-s}$. Then
$$
Q^{n}(e^{r\alpha+s\beta+n\phi})^{\ast}%
$$
is a nontrivial singular vector in the contragradient module $(\mathcal{F}%
_{r,s})^{\ast}$ which implies that there is a vector $v_{n,p}$ of weight
$h+np$ such that
$$
\langle Q^{n}(e^{r\alpha+s\beta+n\phi})^{\ast},v_{n,p}\rangle\neq0.
$$

Since
$$
\langle Q^{n}(e^{r\alpha+s\beta+n\phi})^{\ast},v_{n,p}\rangle\ =(-1)^{n}%
\langle(e^{r\alpha+s\beta+n\phi})^{\ast},Q^{n}v_{n,p}\rangle
$$
we have that $Q^{n}v_{n,p}$ is proportional to $e^{r\alpha+s\beta+n\phi}$. Let
$Z_{n}=\mbox{Ker}_{\mathcal{F}_{r,s}}Q^{n}$. Then $Z_{n}$ is a non-trival
submodule of $\mathcal{F}_{r,s}$ (in particular, $v_{n-1,p}\in Z_{n}$). We
have that
$$
v_{n,p}+Z_{n}\quad\mbox{is a singular vector in }\quad\frac{\mathcal{F}_{r,s}%
}{Z_{n}}.
$$
Using character identity (\ref{decomp}) we get the following filtration:
\bea
\label{formula} \mathcal{F}_{r,s}=\cup_{n\geq0}Z_{n},\quad Z_{n}/Z_{n-1}\cong L^{\HVir%
}(c_{L},0,c_{L,I},h+np,h_{I}).
\eea
The proof follows.
\end{proof}

\begin{remark}
A version of Proposition \ref{struktura} appeared in Theorem 2.4 of \cite{JJ}. In some special cases, the authors obtained the filtration (\ref{formula}) by a different method which was not based on screening operators.
\end{remark}

Finally, we shall now describe the structure of $M(1)$ as a
$L^{\HVir}(c_{L},c_{L,I})$--module.

\begin{corollary}
As a $L^{\HVir}(c_{L},c_{L,I})$--module $M(1)$ is generated by
$\mathbf{1}$ and a family of cosingular vectors $\{v_{n}:n\geq1\}$, such that
$$
Q^{n}v_{n}=e^{-n\frac{\alpha+\beta}{c_{L,I}}}.
$$
In particular,
$$
L^{\HVir}(c_{L},c_{L,I})=\mbox{Ker}_{M(1)}Q.
$$

\end{corollary}

\begin{remark}
It is a standard  technique in the theory of $\mathcal{W}$--algebras and vertex
algebras to realize vertex algebras as kernels of screening operators inside of the  free boson vertex algebras. In particular, in the case of the single boson,
the $\mathcal{W}$--algebra $\mathcal{W}(2,2 p-1)$ and the Virasoro vertex algebra of central charge of $(1,p)$--models are realized as kernels of certain screening
operators (cf. \cite{A-2003}).
\end{remark}

\section{Intertwining operators from the free-field realization}
\label{intertwining}
For a vertex operator algebra $V$ and their modules $M_1, M_2, M_3$, let $${\Y} { M_3 \choose M_1 \ \ \ M_2 } $$ denotes the vector space of all intertwining operators of type
$ { M_3 \choose M_1 \ \ \ M_2 }$ (cf. \cite{FHL}).
In the vertex algebra theory, it is important to construct intertwining operators.
In this section we shall apply the free-field realization from the previous section   to construct
intertwining operators for the twisted Heisenberg-Virasoro Lie algebra ${\HVir}$.

Next we shall use this result to prove reducibility of tensor product of ${\HVir}$--modules
$V_{\alpha,\beta,F}^{\prime} \otimes L^{\HVir}(c_{L},0,c_{L,I},h,h_{I}) $.

Let $V_{\alpha,\beta,F}^{\prime}$ denotes an irreducible $\HVir$-module
from the intermediate series, i.e.,\ $V_{\alpha,\beta,F}^{\prime}%
=\bigoplus\limits_{n\in\mathbb{Z}}\mathbb{C}v_{n}$ with the action
\begin{align*}
L(k)v_{n}  &  =-\left(  n+\alpha+\beta+k\beta\right)  v_{n+k},\\
I(k)v_{n}  &  =Fv_{n+k},\quad\left\{  C_{L},C_{LI},C_{I}\right\}  v_{n}=0.
\end{align*}
Then $V_{\alpha,\beta,F}^{\prime}\otimes L^{\HVir}(c_{L},0,c_{L,I},h,h_{I})$ is a
$\HVir$-module with infinite-dimensional weight spaces (see \cite{R1}).
The tensor product $V_{\alpha,\beta,F}^{\prime}\otimes L^{\HVir}(c_{L},0,c_{L,I},h,h_{I})$ is closely related with intertwining operators for vertex algebras. Let us explain this relation in more details.

Suppose that there exists a non-trivial intertwining operator $\Y$ of the type
$$
{\binom{L^{\HVir}(c_{L},0,c_{L,I},h^{\prime\prime},h_{I}^{\prime\prime
})}{L^{\HVir}(c_{L},0,c_{L,I},h,h_{I})\ \ L^{\HVir}(c_{L}%
,0,c_{L,I},h^{\prime},h_{I}^{\prime})}.}%
$$
Let $v$ be the highest weight vector in $L^{\HVir}(c_{L},0,c_{L,I}%
,h,h_{I}),$ and let $\Y(v,z)=z^{-\alpha}\sum_{n\in{\mathbb{Z}}}v_{n}z^{-n-1}$, where
$$ \alpha = h + h^{\prime} -h^{\prime\prime}.$$
Using commutator formulas for vertex operators, one can show that the
components ${v_{n}} $ of intertwining operator ${\Y}(v,z) $ span a ${\HVir}$--module isomorphic to $V_{\alpha,1-h,h_{I}}^{\prime}.$
In this way we get a non-trivial $\HVir$-homomorphism
$$
\varphi:V_{\alpha,1-h,h_{I}}^{\prime}\otimes L^{\HVir}(c_{L}%
,0,c_{L,I},h^{\prime},h_{I}^{\prime})\rightarrow L^{\HVir}%
(c_{L},0,c_{L,I},h^{\prime\prime},h_{I}^{\prime\prime})
$$%
$$
\varphi\left(  v_{n}\otimes w^{\prime}\right)  =v_{n}w^{\prime}
$$
where $w^{\prime} \in  L^{\HVir}(c_{L},0,c_{L,I},h^{\prime},h_{I}^{\prime})$. Comparing the dimensions of weight spaces, we conclude that
$\ker\varphi$ is nontrivial and, in particular, that $V_{\alpha,1-h,h_{I}}^{\prime
}\otimes L^{\HVir}(c_{L},0,c_{L,I},h^{\prime},h_{I}^{\prime})$ is reducible.

The proof of the following proposition is standard.
\begin{proposition}
\label{hi}
Let
$$
d= \dim {\Y} {\binom{L^{\HVir}(c_{L},0,c_{L,I},h^{\prime\prime},h_{I}^{\prime\prime
})}{L^{\HVir}(c_{L},0,c_{L,I},h,h_{I})\ \ L^{\HVir}(c_{L}%
,0,c_{L,I},h^{\prime},h_{I}^{\prime})}}.
$$
Then $d=0$ or $1$.
If $d=1$, then  $ h_{I}^{\prime\prime} = h_{I}+h_{I}^{\prime} $.
\end{proposition}
\begin{proof}
Let $v$, $v'$ and $v''$  be  highest weight vectors in $L^{\HVir}(c_{L},0,c_{L,I},h,h_{I})$, $L^{\HVir}(c_{L},0,c_{L,I},h',h'_{I})$ and $L^{\HVir}(c_{L},0,c_{L,I},h'',h''_{I})$ respectively. Let $\Y(\cdot, z)$ be the intertwining operator of the above type. Let
$$ \Y(v,z) = \sum _{ n \in \alpha + {\Bbb Z} } v_{n} z ^{-n-1} \quad (\alpha = h+ h' - h'').$$ Then there is $n_0 \in {\alpha} + {\Z}$ such that
\bea \label{jedinstvenost-0} v_{n_0} v' = \nu v'' \quad (\nu \ne 0) \eea
where $n_0 = h + h' - h''$. We have
$$ I(0) (v_{n_0} v') = \nu  (h_I + h'_I) v'' = \nu h''_I v'' . $$
Is it easy to see (for instance see \cite{Li}, Proposition 7.3.5 ) that $\mathcal{Y}$ is uniquely determined (up to a scalar factor) by (\ref{jedinstvenost-0}). The proof follows.
\end{proof}

\vskip 5mm
Now we shall apply this concept on the intertwining operators constructed using the free-field realization.
First we shall use the standard fusion rules result for the Heisenberg vertex algebra $M(1)$.

\begin{proposition}
\label{prop-int-1} Assume that $(r_{1},s_{1}),(r_{2},s_{2}),(r_{3},s_{3}%
)\in\mathbb{C}^{2}$. Then in the category of $M(1)$--modules:
$$
\dim {\Y}{\binom{\mathcal{F}_{r_{3},s_{3}}}{\mathcal{F}_{r_{1},s_{1}%
}\ \ \mathcal{F}_{r_{2},s_{2}}}}=\delta_{r_{1}+r_{2},r_{3}}\delta_{s_{1}%
+s_{2},s_{3}}.
$$
In particular, there exists a unique non-trivial intertwining operator
$\Y(\cdot,z)$ of the type ${\binom{\mathcal{F}_{r_{1}+r_{2},s_{1}+s_{2}}%
}{\mathcal{F}_{r_{1},s_{1}}\ \ \mathcal{F}_{r_{2},s_{2}}}}$ such that
$\Y(v,z)=\sum_{n\in n_{0}+{\mathbb{Z}}}v_{n}z^{-n-1}%
,\ v\in\mathcal{F}_{r_{1},s_{1}}$ and that
$$
(e^{r_{1}\alpha+s_{1}\beta})_{n_{0}-1}e^{r_{2}\alpha+s_{2}\beta}=e^{(r_{1}%
+r_{2})\alpha+(s_{1}+s_{2})\beta}%
$$
where $n_{0}=\Delta_{r_{1},s_{1}}+\Delta_{r_{2},s_{2}}-\Delta_{r_{1}%
+r_{2},s_{1}+s_{2}}$.
\end{proposition}

By restricting the above intertwining operators we get the following
intertwining operator in the category of $\HVir$--modules.

\begin{proposition}
\label{prop-int-2} Assume that
$$
(h,h_{I})=(\Delta_{r_{1},s_{1}},r_{1}-s_{1}),(h^{\prime},h_{I}^{\prime
})=(\Delta_{r_{2},s_{2}},r_{2}-s_{2})\in\mathbb{C}^{2}%
$$
such that
$$
\frac{h_{I}}{c_{L,I}}-1,\frac{h_{I}^{\prime}}{c_{L,I}}-1,\frac{h_{I}%
+h_{I}^{\prime}}{c_{L,I}}-1\ \notin{\mathbb{Z}}_{>0}.
$$
Then there is a non-trivial intertwining operator of the type
$$
{\binom{L^{\HVir}(c_{L},0,c_{L,I},h^{\prime\prime},h_{I}+h_{I}^{\prime
})}{L^{\HVir}(c_{L},0,c_{L,I},h,h_{I})\ \ L^{\HVir}(c_{L}%
,0,c_{L,I},h^{\prime},h_{I}^{\prime})}}%
$$
where
\begin{multline*}
h^{\prime\prime}=\Delta_{r_{1}+r_{2},s_{1}+s_{2}}=
\left(  h_{I}^{\prime}+h_{I}-c_{L,I}\right)  \left(  \frac{h^{\prime}}%
{h_{I}^{\prime}-c_{L,I}}+\frac{h}{h_{I}-c_{L,I}}\right) \\ - h_{I}h_{I}^{\prime
}\left(  \frac{1}{h_{I}^{\prime}-c_{L,I}}+\frac{1}{h_{I}-c_{L,I}}\right)
\frac{c_{L}-2}{24c_{L,I}}%
\end{multline*}

In particular, the $\HVir$--module $V_{\alpha
,\beta,F}^{\prime}\otimes L^{\HVir}(c_{L},0,c_{L,I},h',h'_{I})$ is reducible where
$$\alpha = h + h' - h'', \ \beta = 1-h, \ F = h_I.$$
\end{proposition}

\begin{corollary}
\label{intertw}Assume that
$$
(h,h_{I})=(\Delta_{r_{1},s_{1}},r_{1}-s_{1}),(h^{\prime},h_{I}^{\prime
})=(\Delta_{r_{2},s_{2}},r_{2}-s_{2})\in\mathbb{C}^{2}%
$$
and that there are $p,q\in{\mathbb{Z}}_{>0}$, $q\leq p$ such that
$$
\frac{h_{I}}{c_{L,I}}-1=-q,\ \frac{h_{I}^{\prime}}{c_{L,I}}-1=p.
$$
Then there is a non-trivial intertwining operator of the type
$$
{\binom{L^{\HVir}(c_{L},0,c_{L,I},h^{\prime\prime},h_{I}+h_{I}^{\prime
})}{L^{\HVir}(c_{L},0,c_{L,I},h,h_{I})\ \ L^{\HVir}(c_{L}%
,0,c_{L,I},h^{\prime},h_{I}^{\prime})}}%
$$
where
\begin{multline*}
h^{\prime\prime}=\Delta_{r_{2}-r_{1},s_{2}-s_{1}}=
\left(  h_{I}^{\prime}-h_{I}-c_{L,I}\right)  \left(  \frac{h^{\prime}}%
{h_{I}^{\prime}-c_{L,I}}-\frac{h}{h_{I}-c_{L,I}}\right) \\ +h_{I}\left(
\frac{h_{I}^{\prime}}{h_{I}^{\prime}-c_{L,I}}-\frac{2h_{I}-h_{I}^{\prime}%
}{h_{I}-c_{L,I}}\right)  \frac{c_{L}-2}{24c_{L,I}}=\\
=\left(  p+q-1\right)  \left(  \frac{h^{\prime}}{p}+\frac{h}{q}\right)
+\left(  1-p\right)  \left(  1-q\right)  \left(  \frac{1}{p}+\frac{1}%
{q}\right)  \frac{c_{L}-2}{24}%
\end{multline*}

In particular, the $\HVir$--module $V_{\alpha
,\beta,F}^{\prime}\otimes L^{\HVir}(c_{L},0,c_{L,I},h',h''_{I})$ is reducible where
$$\alpha = h + h' - h'', \ \beta = 1-h, \ F = h_I.$$
\end{corollary}

\begin{proof}
By using Proposition \ref{prop-int-1} we get the non-trivial intertwining operator
$\Y(\cdot,z)$ of the type
$$
{\binom{L^{\HVir}(c_{L},0,c_{L,I},\Delta_{r_{2},s_{2}},-h_{I}^{\prime
}+2c_{L,I})}{L^{\HVir}(c_{L},0,c_{L,I},h,h_{I})\ \ L^{\HVir}%
(c_{L},0,c_{L,I},\Delta_{r_{2}-r_{1},s_{2}-s_{1}},-h_{I}-h_{I}^{\prime
}+2c_{L,I})}}.
$$
Then the adjoint intertwining operator $\Y^{\ast}(\cdot,z)$ is a non-trivial
intertwining operator of type
$$
{\binom{L^{\HVir}(c_{L},0,c_{L,I},\Delta_{r_{2}-r_{1},s_{2}-s_{1}}%
,h_{I}+h_{I}^{\prime})}{L^{\HVir}(c_{L},0,c_{L,I},h,h_{I}%
)\ \ L^{\HVir}(c_{L},0,c_{L,I},\Delta_{r_{2},s_{2}},h_{I}^{\prime})}}.
$$
The proof follows.
\end{proof}

\section{Schur polynomials and singular vectors}

In the theory of Fock spaces for the Virasoro algebra, the  singular vectors can be
expressed using Schur polynomials or Jack polynomials (see \cite{S-jack} for relation between Virasoro singular vectors and Jack polynomials). But these expressions don't give a formula for the  Virasoro singular vectors in the Verma modules in the usual PBW basis. In this section we will see that in the case of the twisted Heisenberg-Virasoro algebra, this approach can give explicit formulas for singular vectors in the PBW basis. Our construction shall use the free-field realization of highest weight modules from Section \ref{free}.

Define the Schur polynomials $S_{r}(x_{1},x_{2},\cdots)$ in variables
$x_{1},x_{2},\cdots$ by the following equation:
\begin{equation}
\exp\left(  \sum_{n=1}^{\infty}\frac{x_{n}}{n}y^{n}\right)  =\sum
_{r=0}^{\infty}S_{r}(x_{1},x_{2},\cdots)y^{r}. \label{eschurd}%
\end{equation}

We shall also use the following formula for Schur polynomials:
\begin{equation}
S_{r}(x_{1},x_{2},\cdots)=\frac{1}{r!}\det\left(
\begin{array}
[c]{ccccc}%
x_{1} & x_{2} & \cdots & \  & x_{r}\\
-r+1 & x_{1} & x_{2} & \cdots & x_{r-1}\\
0 & -r+2 & x_{1} & \cdots & x_{r-2}\\
\vdots & \ddots & \ddots & \ddots & \vdots\\
0 & \cdots & 0 & -1 & x_{1}
\end{array}
\right)  \label{det-schur}%
\end{equation}

Schur polynomials naturally appear in the formulas for vertex operators for lattice
vertex algebras (cf. \cite{K}, \cite{LL} ) and they are useful in
the representation theory of $\mathcal{W}$--algebras (for one application see
\cite{A-2003}). We shall here only recall that for $\gamma\in D$, $\delta\in D
^{0} $, $\langle\gamma, \delta\rangle= -n$, $n \in{\mathbb{Z} _{> 0} }$ we
have
$$
e^{\gamma}_{i-1} e ^{\delta} = \nu S_{n-i} (\gamma(-1), \gamma(-2), \cdots) e
^{ \gamma+ \delta}\quad(\nu\in\{\pm1\} ).
$$

Let $V^{\HVir}(c_{L},0,c_{L,I},h,h_{I})$ denote the Verma module for the twisted
Hei\-sen\-berg-Virasoro Lie algebra. Let $v_{h,h_{I}}$ be its highest-weight vector.

\begin{lemma}
\label{pom-1} Assume that $v\in\mathcal{I}\subset V^{\HVir}(c_{L},0,c_{L,I},h,h_{I})$
is such that $\Phi(v)\in\mathcal{F}_{r,s}$ is a non-trivial singular vector.
Then $v$ is a singular vector in $V^{\HVir}(c_{L},0,\allowbreak c_{L,I},h,h_{I})$.
\end{lemma}

\begin{proof}
Assume that $v$ is not a singular vector. Then there is $n_{0} > 0$ such that
$L(n_{0}) v \ne0$. But $L(n_{0}) v \in \mathcal{I} $ and therefore $L(n_{0}) \Phi(v) =
\Phi( L(n_{0}) v) \ne0$. This is a contradiction. Proof follows.
\end{proof}

Our main result is:

\begin{theorem}
Assume that $c_{L,I}\neq0$ and $p=\frac{h_{I}}{c_{L,I}}-1\in{\mathbb{Z}}%
_{>0}.$ Then
$$
S_{p}\left(-\frac{I(-1)}{c_{L,I}},-\frac{I(-2)}{c_{L,I}},\dots,-\frac{I(-p)}{c_{L,I}}\right)v_{h,h_{I}}%
$$
is a singular vector in the Verma module $V^{\HVir}(c_{L},0,c_{L,I},h,h_{I})$ of
conformal weight $p$.
\end{theorem}

\begin{proof}
Take now (unique) $\lambda,\mu,r,s\in{\mathbb{C}}$ such that relations
(\ref{par-1}) and (\ref{par-2}) hold. Then $M:=U(\HVir).e^{r\alpha
+s\beta}$ is a highest weight module with the highest weight $(h,h_{I})$.

Then
$$
(-\frac{\alpha+\beta}{c_{L,I}},r\alpha+s\beta+\frac{\alpha+\beta}{c_{L,I}%
})=-p-1.
$$
Therefore,
$$
Qe^{r\alpha+s\beta+\frac{\alpha+\beta}{c_{L,I}}}=S_{p}\left(-\frac{I(-1)}{c_{L,I}%
},\dots,-\frac{I(-p)}{c_{L,I}}\right)e^{r\alpha+s\beta}%
$$
is a singular vector in $M$. Lemma \ref{pom-1} gives that $S_{p}%
(-\frac{I(-1)}{c_{L,I}},\dots,-\frac{I(-p)}{c_{L,I}})v_{h,h_{I}}$
is a singular vector in $V^{\HVir}(c_{L},0,c_{L,I},h,h_{I})$. Proof follows.
\end{proof}

Now  we shall present a formula for singular vectors in the case $\frac{h_{I}}{c_{L,I}}-1\in-{\N}.$

\begin{theorem}
Let $\frac{h_{I}}{c_{L,I}}-1=-p\in{\N}$ and let $\Lambda\in
V^{\HVir}(c_{L},c_{L,I},h,h_{I})_{p}$ be a singular vector. Then
\[
\Lambda=\sum_{i=0}^{p-1}S_{i}\left(  \frac{I(-1)}{c_{L,I}},\ldots,\frac{I(-i)%
}{c_{L,I}}\right)  L(i-p)v+x_{p}v,
\]
for some $x_{p}v\in\mathcal{I}_{p}$.
\end{theorem}

\begin{proof}
It was proven in (\cite{R1}) that $\Lambda$ is linear in $L$'s, i.e.\
\[
\Lambda=L(-p)+\sum_{i=1}^{p-1}x_{i}L(i-p)v+x_{p},\text{ where }x_{i}%
\in\mathcal{I}_{-i}.
\]

We shall prove our assertion by induction.
Suppose that $$x_{i}=S_{i}\left(  \frac{I(-1)}{c_{L,I}},\ldots,\frac{I(-i)%
}{c_{L,I}}\right)  \quad  \mbox{for} \ \  i=1,\ldots,k-1.$$ Using the fact that $I(i)$'s
commute and that $I(i)\Lambda=0$ for $i>0$ we have
\begin{multline*}
0=I(p-k)\Lambda=\\
(p-k)\left(  \sum_{i=0}^{k-1}S_{i}\left(  \frac{I(-1)}{c_{L,I}}%
,\ldots,\frac{I(-i)}{c_{L,I}}\right)  I(i-k)v+(h_{I}+(p-k-1)c_{L,I}%
x_{k}v\right)  .
\end{multline*}
Since $h_{I}=-(p-1)c_{L,I}$ we get $x_{k}=\frac{1}{kc_{L,I}}\sum_{i=0}%
^{k-1}S_{i}\left(  \frac{I(-1)}{c_{L,I}},\ldots,\frac{I(-i)}{c_{L,I}}\right)
I(i-k)$ and the right side is exactly the recursive relation one gets from the
Laplace expansion in (\ref{det-schur}). Therefore, $x_{k}=S_{k}\left(
\frac{I(-1)}{c_{L,I}},\ldots,\frac{I(-k)}{c_{L,I}}\right)  $.
\end{proof}

\section{An application of singular vectors to the irreducibility of tensor products}

We shall now apply our formula for singular vectors which uses Schur polynomials to the problem of irreducibility of the tensor products
$V_{\alpha,\beta,F}^{\prime} \otimes L^{\HVir}(c_{L},0,c_{L,I},h,h_{I}) $. Similar irreducibility problems for the Virasoro algebra and the twisted Heisenberg-Virasoro algebra were studied in the papers \cite{CGZ}, \cite{LZ} \cite{R-ART}, \cite{R1}, \cite{Zh}. Our irreducibility result is a certain refinement of these results.

Let us discuss the irreducibility of $V_{\alpha,\beta,F}^{\prime}\otimes
L^{\HVir}(c_{L},0,c_{L,I},h,h_{I})$ in general. It was shown in \cite{R1} (see also \cite{LZ} ) that this
module is irreducible if and only if \bea \label{cond-irr} U(\HVir)(v_{n}\otimes v)\subseteq
U(\HVir)(v_{n+1}\otimes v) \quad \forall n\in\mathbb{Z}. \eea One can obtain
this inclusion using explicit formulas for  singular vectors in $V^{\HVir}(c_{L},0,c_{L,I}%
,h,h_{I}).$  We shall first consider the case $h_I /c_{L,I} - 1 = p \in {\N}$.  Then the maximal submodule in $V^{\HVir}(c_{L},0,c_{L,I},h,h_{I})$ is generated by the singular vector
$$ \Omega = S_{p}(-\frac{I(-1)}{c_{L,I}},-\frac{I(-2)}{c_{L,I}},\cdots)v_{h,h_{I}}. $$

In \cite{LZ}, the authors introduce  a very useful criterion for studying the irreducibility of tensor products (this approach   is equivalent to  the criterion (\ref{cond-irr}) ). They introduce the linear map $\phi_n : U(\HVir^{-} ) \rightarrow {\Bbb C}$ inductively as follows:
\bea
&& \phi_n (1) = 1 \nonumber  \\
&& \phi_n( I ({-i}) u ) = - F  \phi_n (u) \nonumber \\
&& \phi_n (L( -i) u ) =   (\alpha + \beta + k + i + n - i \beta ) \phi_n (u) \label{rel-lu} \eea
for every $u \in   U(\HVir^{-} ) $ of weight $-k$.

Then by  Lemma 25 of  \cite{LZ} (see also Theorem 1 of \cite{CGZ} for a similar statement in the Virasoro case) we have that $V_{\alpha,\beta,F}^{\prime} \otimes L^{\HVir}(c_{L},0,c_{L,I},h,h_{I}) $ is irreducible if and only if $\phi_n( \Omega) \ne 0$ for every $n \in {\Bbb Z}$.

In our case we have:

\begin{lemma} Let $p=\frac{h_{I}}{c_{L,I}}-1\in{\N}$. For every $n \in {\Z}$ we have
$$ \phi_n (\Omega) = (-1) ^p  {- \frac{F}{ c_{L,I} } \choose p }. $$
\end{lemma}

So we obtain:

\begin{theorem}
\label{-+}
Let $p=\frac{h_{I}}{c_{L,I}}-1\in{\N}.$ The module $V_{\alpha
,\beta,F}^{\prime}\otimes L^{\HVir}(c_{L},0,c_{L,I},h,h_{I})$ is irreducible if and only if   $F\neq(i-p)c_{L,I},$ for $i=1,\ldots,p.$
\end{theorem}

\begin{remark}
The reducibility of certain tensor product modules  $V_{\alpha
,\beta,F}^{\prime}\otimes L^{\HVir}(c_{L},0,c_{L,I},h,h_{I})$ (not all!) can be directly proved by using the intertwining operators constructed in Section \ref{intertwining}.
\end{remark}

We shall now extend the   irreducibility result obtained in Theorem \ref{-+} to the case
$h_I / c_{L,I} -1  \in - {\N}$.

\begin{lemma} \label{pomoc-red}
Let $ \frac{h_I}{c_{L,I}} - 1 = - p \in -{\N}$. Then for every $n \in {\Z}$ we have
\begin{multline*}
\phi_{n}(\Lambda) = \sum_{i=0}^{p-1}\left(  \left(  n+p-i+\alpha
+(1+i-p)\beta\right) (-1)^i  { \frac{F}{c_{L,I}} \choose i}\right)  + \phi_{n}(x_{p})= \\
=(-1) ^{p-1}  \left( { F/ c_{L,I} -1 \choose p- 1}   (\alpha + n + \beta) + (1-\beta) { F / c_{L,I}-2 \choose p-1} \right) +  g_p ( F)
\end{multline*}
for a certain polynomial $g_p \in {\Bbb C}[x]$.

In particular, if $ F / c_{L,I}  \notin \{ 1, \dots, p-1\}$, then for every $n \in {\Z}$, there is a unique $\alpha:=\alpha_n \in {\Bbb C}$  such that $\phi_n (\Lambda) = 0$.
\end{lemma}

\begin{theorem} \label{red-case2} Let $ \frac{h_I}{c_{L,I}} - 1 = - p \in -{\N}$.
\item[(1)]Let $ F / c_{L,I}  \notin \{ 1, \dots, p-1\}$ and let $\alpha_{0} \in {\Bbb C}$ be such that
$$ \phi_{0}(  \Lambda)=  0.$$
Then $V'_{\alpha, \beta, F} \otimes L^{\HVir}(c_{L},0,c_{L,I},h,h_{I}) $ is reducible if and only if $\alpha \equiv  \alpha_{0} \mod \Z$. In this case
 $W^0 = U(\HVir). (v_{0}  \otimes v)$ is the simple submodule  of  $V'_{\alpha, \beta, F} \otimes L^{\HVir}(c_{L},0,c_{L,I},h,h_{I}) $  and
 $V'_{\alpha, \beta, F} \otimes L^{\HVir}(c_{L},0,c_{L,I},h,h_{I})  / W^0 $
 is a highest weight $\HVir$--module $ \widetilde{L}^{\HVir}(c_{L},0,c_{L,I},h'',h''_{I}) $(not necessarily irreducible)
 where $$ h'' = - \alpha_{0}  + h + (1-\beta), \ \ h''_I = F + h_I. $$

\item[(2)] Let $ F / c_{L,I}  \in\{ 2, \dots, p-1\}$.  Then $V'_{\alpha, \beta, F} \otimes L^{\HVir}(c_{L},0,c_{L,I},h,h_{I}) $ is reducible. Moreover, $g_p(i c_{L,I})=0$ for $i=2, \dots, p$.

     \item[(3)] Let $ p>1 $ and  $ F / c_{L,I}  = 1$.  Then $V'_{\alpha, \beta, F} \otimes L^{\HVir}(c_{L},0,c_{L,I},h,h_{I}) $ is reducible if and only if $1 - \beta = \frac{c_L-2}{24}$. Moreover, $g_p(c_{L,I}) = \beta - 1$.
 \end{theorem}
\begin{proof}
The proof of (1) uses the fact that for $\alpha =\alpha_0$ we have $\Phi_0 (\Lambda) = 0$ and $\Phi_n (\Lambda)\ne 0$ for $n \ne 0$. Then the assertion follows by using the arguments from \cite{R1}, \cite{LZ}.

Let us consider the case (2). Using Lemma \ref{pomoc-red} we see that  $V'_{\alpha, \beta, F} \otimes L^{\HVir}(c_{L},0,c_{L,I},h,h_{I}) $ is reducible if and only if $g_p( F) = 0 $. So this tensor product is either irreducible for all $\alpha, \beta$ or it is reducible for all $\alpha, \beta$. Therefore the reducibility in one case  immediately gives that $g_p(F) = 0$, which will imply the reducibility in all cases. For that purpose we use the intertwining operator from Proposition \ref{prop-int-1}
in the case
$$ (\Delta_{r_1,s_1}, r_1-s_1) = (1-\beta, F), \  (\Delta_{r_2,s_2}, r_2-s_2) = (h, h_I), $$
which gives a non-trivial   $L^{\HVir}(c_L, c_{L,I} )$--intertwining operator of type
$$ {\mathcal{F}_{r_1+r_2,s_1+s_2} \choose  V^{\HVir} (c_L,0,c_{L,I}, 1-\beta, F) \ \ L^{\HVir}  (c_L,0,c_{L,I}, h, h_I) }. $$
(Note that here $U(\HVir). e ^{r_1 \alpha + s_1 \beta}$ is isomorphic to the Verma module $V^{\HVir} (c_L,0,\allowbreak c_{L,I},1-\beta, F)$).

This proves the existence of a non-trivial $\HVir$--homomorphism
$$\varphi : V'_{\alpha, \beta, F} \otimes L^{\HVir}(c_{L},0,c_{L,I},h,h_{I}) \rightarrow  \mathcal{F}_{r_1+r_2,s_1+s_2}$$
and the reducibility of $V'_{\alpha, \beta, F} \otimes L^{\HVir}(c_{L},0,c_{L,I},h,h_{I})$. In particular we get $g_p(F) =0$.

Let us prove the assertion (3). From Lemma \ref{pomoc-red} we see that  there is a unique $\beta $ such that $V'_{\alpha, \beta, F} \otimes L^{\HVir}(c_{L},0,c_{L,I},h,h_{I}) $ is reducible and for all other values of $\beta$, the module is irreducible.  Once again we apply the intertwining operators and Proposition \ref{prop-int-1} in the case
$$ (\Delta_{r_1,s_1}, r_1-s_1) = (\frac{c_L-2}{24}, c_{L,I}), \  (\Delta_{r_2,s_2}, r_2-s_2) = (h, h_I). $$
This gives a non-trivial $\HVir$--homomorphism
$$\varphi : V'_{\alpha, \beta, F} \otimes L^{\HVir}(c_{L},0,c_{L,I},h,h_{I}) \rightarrow  \mathcal{F}_{r_1+r_2,s_1+s_2}$$
and proves the reducibility of the module
$V'_{\alpha, \beta, F} \otimes L^{\HVir}(c_{L},0,c_{L,I},h,h_{I})$ for $\alpha = \frac{c_L-2}{24} + h -\Delta_{r_1+r_2,s_1+s_2}$.
Since the reducibility does not depend on the parameter $\alpha$, we get the assertion (3).
\end{proof}

We present the simplest case in the following examples.
\begin{example}
$ \Lambda = (L(-1) + h^{\prime} / c_{L,I} I(-1)  )v $ is a singular vector in the Verma module $ V^{\HVir}(c_L,0,c_{L,I},h^{\prime} ,0)$ and
$ \phi_n(\Lambda) = n + 1 + \alpha - Fh^{\prime} /c_{L,I} $. Therefore, the  module $V^{\prime}_{\alpha,\beta,F}\otimes L^{\HVir}(c_{L},0,c_{L,I},h^{\prime} ,0)$ is reducible if and only if
$ \a \equiv  F h^{\prime} /c_{L,I}  \mod \Z $. If we let $\a = h + h^{\prime} - h^{\prime\prime} $, $\b = 1 - h$ and $ F = h_{I} $ we get the reducibility condition $ h^{\prime\prime} = h + (1 - h_{I}/c_{L,I})h^{\prime} $. Setting $ h_{I}^{\prime} = 0 $, i.e., $r_2 = s_2$ in Proposition \ref{prop-int-2} gives the same $h^{\prime\prime}$.
\end{example}

\begin{example}
An example of singular vector in $V^{\HVir}(c_L,0,c_{L,I},h^{\prime},-c_{L,I})$ is
$$ \Lambda = (L(-2) + \frac{1}{c_{L,I}} I(-1)L(-1) + \frac{c_L+8h^{\prime} -2}{16c_{L,I}}I(-2) + \frac{c_L+24h^{\prime}-2}{48c_{L,I}^2}I(-1)^2)v. $$
Then we have
$ \phi_n(\Lambda) = (n + 2 + \a - \b) - \frac{F}{c_{L,I}}(n + 1 + \a) - \frac{F}{c_{L,I}}\frac{c_L+8h^{\prime}-2}{16} +
\frac{F^2}{c_{L,I}^2} \frac{c_L+24h^{\prime}-2}{48}$.
Let $F = h_{I} = -k c_{L,I}$, $k \in \Zp $. The  module $V^{\prime}_{\alpha,\beta,F}\otimes L^{\HVir}(c_{L},0,c_{L,I},h^{\prime},-c_{L,I})$ is reducible if and only if
 $\alpha = h + h' - h''$  where
$$h = 1-\beta, \quad  h^{\prime\prime} = \frac{k+2}{k+1}h + \frac{k+2}{2}h^{\prime} + \frac{k(k+3)}{k+1} \frac{c_L-2}{48} $$
Again,   $h''=\Delta_{r_1+r_2,s_1+s_2} $ in Proposition \ref{prop-int-2}.
\end{example}

\section{Fusion rules for $L^{\HVir}(c_L, c_{L,I} )$--modules}

In this section we shall consider the fusion ring for $L^{\HVir}(c_L, c_{L,I} )$--modules. In particular, we shall consider the fusion subring generated by $L^{\HVir}(c_L, c_{L,I} )$--modules:
\bea \label{category}  L^{\HVir}(c_L, 0, c_{L,I}, h, h_I,  ) \quad \mbox{
such that} \quad  h_I /c_{L,I} -1 \in {\Z} \setminus \{ 0 \}. \eea

The  irreducibility result from Theorem \ref{-+}  imply the following result on vanishing of  fusion rules:

\begin{corollary} \label{vanishing}
Let $$\frac{h_{I}}{c_{L,I}}-1 \in
{\N}, \quad \frac{h_{I}^{\prime}}{c_{L,I}}-1 \notin \{ -1, -2, \dots, - p\}.$$ Then
$$
\dim {\Y} {\binom{M}{L^{\HVir}(c_{L},0,c_{L,I},h,h_{I})\ \ L^{\HVir%
}(c_{L},0,c_{L,I},h^{\prime},h_{I}^{\prime})}=0}%
$$
for all $h,h^{\prime},c_{L}\in
\mathbb{C}$  and every $L ^{\HVir} (c_L, c_{L,I})$--module $M$.
\end{corollary}

\begin{remark} \label{zero-1} Assume that $M_1$ and $M_2$ are modules over vertex algebra $V$. One can define a tensor product of modules $M_1$ and $M_2$ as an ordered pair $(  T(M_1,M_2), F)$ where $T(M_1,M_2)$ is a $V$--module and $F$ is an intertwining operator of the type
$$
 {\binom{T(M_1, M_2) }{M_1 \ \ M_2}}
$$
with the following universal property:
for every $V$--module $U$ and every intertwining operator $\Y$ of type  ${\binom{U }{M_1 \ \ M_2}} $ there exists a $V$--module homomorphism $\Psi : T(M_1, M_2) \rightarrow U$ such that  $\Y (\cdot, z) = \Psi \circ F(\cdot, z)$. It is a very difficult problem to prove the existence of a tensor product and associativity in a certain suitable category of modules. For details about the tensor product theory see \cite{HLZ} and references therein. In particular it was noted in \cite{Li} that if $V$ is simple and self-dual and if the tensor product is associative, then the tensor product of two simple modules is non-zero.

Our Corollary \ref{vanishing} shows that in the category of $L ^{\HVir} (c_L, c_{L,I})$--modules, the tensor product of modules
$L^{\HVir}(c_{L},0,c_{L,I},h,h_{I}) $ and  $L^{\HVir
}(c_{L},0,c_{L,I},h^{\prime},h_{I}^{\prime})$  can be zero. Our vertex algebra  $L ^{\HVir} (c_L, c_{L,I})$ is simple but it is not self-dual. We still believe that the tensor product in the category of $L ^{\HVir} (c_L, c_{L,I})$--modules is associative.

One can also construct examples such that the tensor product module of two non-zero modules is zero in the category of modules for the $C_2$--cofinite vertex algebra $\mathcal{W}_{p,p'}$ investigated in \cite{AdM-IMRN}.
\end{remark}

\begin{lemma}
\label{jedinstvenost}
Assume that
$
(h,h_{I}), (h',h_{I}^{\prime}) \in\mathbb{C}^{2}%
$
such that $$\frac{h_{I}}{c_{L,I}}-1 = -p \in -{\N}, \ \ \frac{h_{I}^{\prime}}{c_{L,I} }  \notin \{ 1, \dots, p-1 \}. $$
Assume also that there is a non-trivial intertwining operator of the type
$$
{\binom{\widetilde{L}^{\HVir}(c_{L},0,c_{L,I},h^{\prime\prime},h_{I}^{\prime\prime})}{L^{\HVir}(c_{L},0,c_{L,I},h,h_{I})\ \ \widetilde{L}^{\HVir}(c_{L},0,c_{L,I},h^{\prime},h_{I}^{\prime})}}%
$$
where $\widetilde{L}^{\HVir}(c_{L},0,c_{L,I},h^{\prime},h_{I}^{\prime})$ and $\widetilde{L}^{\HVir}(c_{L},0,c_{L,I},h^{\prime \prime},h_{I}^{\prime \prime})$
are certain highest \allowbreak weight $\HVir$--modules with the highest weights $(h,h_{I}), (h',h_{I}^{\prime}) $. Then the highest weight $(h^{\prime \prime},h_{I}^{\prime \prime})$ is uniquely determined.
\end{lemma}
\begin{proof}
Assume that there is a non-trivial intertwining operator $\Y (\cdot, z)$ of the type
$$
{\binom{\widetilde{L}^{\HVir}(c_{L},0,c_{L,I},h^{\prime\prime},h_{I}^{\prime\prime})}{L^{\HVir}(c_{L},0,c_{L,I},h,h_{I})\ \ \widetilde{L}^{\HVir}(c_{L},0,c_{L,I},h^{\prime},h_{I}^{\prime})}}%
$$
Then clearly $h''_I = h_I + h'_I$. Let $\alpha = h + h' - h''$.
Let $v,v', v'' $ be the highest weight vectors in modules $L^{\HVir}(c_{L},0,c_{L,I},h,h_{I})$, $\widetilde{L}^{\HVir}(c_{L},0,c_{L,I},h',h'_{I})$ and $\widetilde{L}^{\HVir}(c_{L},0,c_{L,I},h'',h''_{I})$ respectively.
Let $\mathcal{Y}(v,z) = \sum_{n \in {\Z} + \alpha} v_n z ^{-n-1}$. Then there is $n_0 \in \alpha + {\Z}$ such that
$$ v_{n_0} v' = \nu v'' \quad (\nu \ne 0). $$
As usual, we shall denote $v_{n_0}$ as $o(v)$.
By using singular vector in the Verma module $V^{\HVir}(c_{L},0,c_{L,I},h,h_{I})$ and the proof of Theorem \ref{red-case2} (1) we get the following relation
\bea \left( (L(0) - h) o(v) { \frac{I(0)}{c_{L,I}}  -1 \choose p-1 } - o(v)  L(0)  { \frac{I(0)}{c_{L,I}} -2  \choose p-1 } + o(v) \widetilde{g_p} (I(0) ) \right) v' =0. \nonumber \\ \label{fund-relation} \eea
for a certain polynomial $\widetilde{g_p} \in {\Bbb C}[x]$.
This shows that
\bea ( h'' - h) { \frac{h^{\prime} _I }{c_{L,I}}  -1 \choose p-1 } -   h ^{\prime} { \frac{h^{\prime} _I }{c_{L,I}} -2  \choose p-1 } +  \widetilde{g_p} (h^{\prime} _I )  =0. \nonumber  \eea
and therefore $h''$ is uniquely determined. The proof follows.
\end{proof}

\begin{remark}
By using Frenkel-Zhu's bimodules (cf. \cite{FZ}) relation (\ref{fund-relation}) can be obtained in the form
$$ (\omega -h ) * v * { \frac{I}{c_{L,I}}  -1 \choose p-1 } - v*\omega * { \frac{I}{c_{L,I}} -2  \choose p-1 } + v* \widetilde{g_p}(I) = 0. $$
We leave the details to an interested reader.
\end{remark}

\begin{theorem}
\label{uniq}
Assume that
$$
(h,h_{I})=(\Delta_{r_{1},s_{1}},r_{1}-s_{1}), \ \ (h^{\prime},h_{I}^{\prime
})=(\Delta_{r_{2},s_{2}},r_{2}-s_{2})\in\mathbb{C}^{2}%
$$
such that $\frac{h_{I}}{c_{L,I}}-1,\frac{h_{I}^{\prime}}{c_{L,I}}-1\in -{\N}$.
Then an intertwining operator of the type
$$
{\binom{L^{\HVir}(c_{L},0,c_{L,I},h^{\prime\prime},h_{I}^{\prime\prime})}{L^{\HVir}(c_{L},0,c_{L,I},h,h_{I})\ \ L^{\HVir}(c_{L},0,c_{L,I},h^{\prime},h_{I}^{\prime})}}%
$$
exists if and only if $ (h^{\prime\prime},h_{I}^{\prime\prime}) = (\Delta_{r_{1}+r_{2},s_{1}+s_{2}},h_{I}+h_{I}^{\prime})  $
\end{theorem}

\begin{proof}
The existence of the above intertwining operators was proved in Proposition \ref{prop-int-2}. The uniqueness follows from Lemma \ref{jedinstvenost}.
\end{proof}

If a nontrivial intertwining operator of the type
$ {\binom{M_3}{M_1\ \ M_2}} $
exists, then there exist a nontrivial transpose operator of the type
$ {\binom{M_3}{M_2\ \ M_1}} $
and a nontrivial adjoint operator of the type
$ {\binom{M_{2}^{*}}{M_1\ \ M_{2}^{*}}} $.
Combining this with the previous results we get the following  result on fusion rules.

\begin{theorem}
\label{classif}
Let $ (h,h_{I})=(\Delta_{r_{1},s_{1}},r_{1}-s_{1}),(h^{\prime},h_{I}^{\prime})=(\Delta_{r_{2},s_{2}},r_{2}-s_{2})\in\mathbb{C}^{2} $ such that
$$
\frac{h_{I}}{c_{L,I}}-1=q,\ \frac{h_{I}^{\prime}}{c_{L,I}}-1=p,\ p,q \in {\Z} \setminus \{0 \}.  $$
Let
$$
d = \dim {\Y} {\binom{L^{\HVir}(c_{L},0,c_{L,I},h^{\prime\prime},h_{I}^{\prime\prime})}{L^{\HVir}(c_{L},0,c_{L,I},h,h_{I})\ \ L^{\HVir}(c_{L}%
,0,c_{L,I},h^{\prime},h_{I}^{\prime})}}.
$$
Then $ d=1 $ if and only if $h_{I}^{\prime\prime} = h_{I} + h_{I}^{\prime} $ and if one of the following holds
\begin{description}
\item[(i)] $ p,q <0 $ and $ h^{\prime\prime} = \Delta_{r_{1}+r_{2},s_{1}+s_{2}}=\left(  1+p+q\right)  \left(  \frac{h^{\prime}}{p}+\frac{h}{q}\right)
-$\newline$\left(  1+p\right)  \left(  1+q\right)  \left(  \frac{1}{p}%
+\frac{1}{q}\right)  \frac{c_{L}-2}{24} $;
\item[(ii)] $ 1 \leq -q \leq p $ and $  h^{\prime\prime} = \Delta_{r_{2}-r_{1},s_{2}-s_{1}}=\left(  1-p+q\right)  \left(  \frac{h}{q}-\frac{h^{\prime}}%
{p}\right)  +\left(  1-p\right)  \left(  1+q\right)  \left(  \frac{1}{p}%
-\frac{1}{q}\right)  \frac{c_{L}-2}{24} $;
\item[(iii)]  $ 1 \leq -p \leq q $ and $  h^{\prime\prime} = \Delta_{r_{2}-r_{1},s_{2}-s_{1}}=\left(  1-p+q\right)  \left(  \frac{h}{q}-\frac{h^{\prime}}%
{p}\right)  +\left(  1-p\right)  \left(  1+q\right)  \left(  \frac{1}{p}%
-\frac{1}{q}\right)  \frac{c_{L}-2}{24} $.
\end{description}
$ d=0 $ otherwise.
\end{theorem}

\begin{proof}
$ h_{I}^{\prime\prime} = h_{I} + h_{I}^{\prime} $ by Proposition \ref{hi}. We continue case by case.
\begin{description}
\item[a]
Let $ p,q<0 $. Then by Theorem \ref{uniq} $d = 1$ if and only if $ h^{\prime\prime} = \Delta_{r_{1}+r_{2},s_{1}+s_{2}} $.
\item[b]
Let $ p, -q >0 $. Assume that (ii) holds, then $d = 1$ by Corollary \ref{intertw}. If $ 1 \leq p < -q $, then $ d=0 $ by Corollary \ref{vanishing}. Suppose that $ d=1 $, $ -q \leq p $ and $ h^{\prime\prime} \neq \Delta_{r_{1}+r_{2},s_{1}+s_{2}} $. Taking in account an adjoint intertwining operator we get a contradiction with Theorem \ref{uniq}.
\item[c]
Let $ -p,q >0 $. One can show that $ d=1 $ if and only if (iii) holds by using transposed operators and case b.
\item[d]
Let $ p,q>0 $. Then $ d=0 $ by Corollary \ref{vanishing}.
\end{description}
This completes the proof.
\end{proof}

\begin{remark}
The fusion rules for a larger category of $\HVir$--modules are more complicated. Let $h_I/c_{L,I} -1 = -p$, $p \ge 1$. Previous theorem shows that the "tensor product functor" $T(L^{\HVir}(c_L, 0, c_{L,I}, h, h_I ), \cdot)$ is a permutation of irreducible modules (\ref{category}). But this is not the case in general. Let $p \ge 2$. By  using Propositions \ref{self-dual} and \ref{prop-int-1} we see that there are infinitely many non-isomorphic irreducible modules $M$ such that the space of intertwining operators
$$ {M \choose L^{\HVir}(c_L, 0, c_{L,I}, \frac{c_L-2}{24}, c_{L,I} ) \quad  L^{\HVir}(c_L, 0, c_{L,I}, h, h_I ) } $$
is non-trivial.
Moreover, as in Remark \ref{zero-1},  we see that the tensor product $T( L^{\HVir}(c_L, 0, c_{L,I}, h', c_{L,I} ), L^{\HVir}(c_L, 0, c_{L,I}, h, h_I ) ) = 0$ if $h' \ne \frac{c_L-2}{24}$.

These arguments show  that modules
$L^{\HVir}(c_L, 0, c_{L,I}, h, h_I )$ are never  simple currents if $p \ge 2$.
\end{remark}

\section{Free field realization of the vertex algebra $W(2,2)$}

In this section we shall present a free-field realization of the vertex  algebra $W(2,2)$. In fact we shall embed the   vertex algebra $W(2,2)$ inside the  Heisenberg-Virasoro vertex algebra investigated in previous sections. As a consequence, we shall prove that the singular vectors that appeared in the analysis of the Verma modules for the twisted Heisenberg-Virasoro algebra  become singular vectors in the Verma modules for the $W(2,2)$ algebra.

Recall that $W(2,2)$--algebra is an infinite-dimensional Lie algebra with a basis $\{W (n),L(n), C ,C_{W},:n\in{\Z} \}$ over $\mathbb{C}$ and
commutation relations
\begin{gather*}
\left[  L(n) ,L (m) \right]  =\left(  n-m\right)  L (n+m)+\delta_{n,-m}%
\frac{n^{3}-n}{12}C,\\
\left[  L (n) ,W (m )\right]  =\left(  n-m\right)  W (n+m )+\delta_{n,-m}%
\frac{n^{3}-n}{12}C_{W},\\
\left[  W (n),W( m ) \right] = 0  ,\\
C, C_{W} \quad\mbox{ are in the center of } \ W(2,2).
\end{gather*}
Let $L^{W(2,2)} (c_{L}, c_{W})$ be a universal vertex algebra associated
to $W(2,2) $ (cf. \cite{DZ}, \cite{R1}). Recall that $L^{W(2,2)} (c_{L},
c_{W})$ is generated by the fields
$$
L(z) = Y(\omega, z) = \sum_{n \in {\Z}} L(n) z ^{-n-2} , \ W(z) =
Y(w,z) = \sum_{n\in{\Z}} W(n) z^{-n-2}.
$$
Here $\omega= L(-2) \mathbf{1}$ and $w = W(-2) \mathbf{1}$.

Let $V^{W(2,2)} (c_{L}, c_{W}, h, h_{W})$ be the Verma module for $W(2,2)$
with the highest weight $(h, h_{W})$.

\begin{theorem}
There is a non-trivial homomorphism of vertex algebras
\begin{align}
\Psi:L^{W(2,2)}(c_{L},c_{W})  &  \rightarrow L^{\HVir}(c_{L}%
,c_{L,I})\nonumber\\
\omega &  \mapsto L(-2)\mathbf{1}\nonumber\\
w  &  \mapsto(I(-1)^{2}+2c_{L,I}I(-2))\mathbf{1}\nonumber
\end{align}
where
$$
c_{W}=-24c_{L,I} ^2.
$$

\end{theorem}

\begin{proof}
By direct calculation we get:
$$
L(0)w=2w,\ \ L(1)w=0,\ \ L(2)w=-12c_{L,I} ^2 \mathbf{1}=\frac{c_{W}}{2}\mathbf{1}.
$$
Now the commutator formula  gives that the components of the fields $L(z)$, $W(z)$ satisfy the commutation relations for the
$W(2,2)$-algebra. This proves the assertion.
\end{proof}

By previous theorem every $L^{\HVir}(c_{L},c_{L,I})$--module becomes a
$L^{W(2,2)}(c_{L},c_{W})$--module. In particular, the Verma module for the twisted
Heisenberg-Virasoro Lie algebra $V^{\HVir}(c_{L},0,c_{L,I},h,h_{I})$ is a
$L^{W(2,2)}(c_{L},c_{W})$--module and $v_{h,h_{I}}$ is the $W(2,2)$ highest
weight vector such that
$$
L(0)v_{h,h_{I}}=hv_{h,h_{I}},\quad W(0)v_{h,h_{I}}=h_{W}v_{h,h_{I}}%
$$
where $h_{W}=h_{I}(h_{I}-2c_{L,I})$.

This construction gives a non-trivial $W(2,2)$--homomorphism
$$
\Psi:V^{W(2,2)}(c,c_{W},h,h_{W})\rightarrow V^{\HVir}(c_{L},0,c_{L,I},h,h_{I}).
$$

Let $\mathcal{W} = {\mathbb{C}}[ W(-1), W(-2), \dots] v_{h,h_{W}} \subset
V^{W(2,2)} (c,c_{W}, h, h_{W})$.

\begin{lemma}
Assume that $\frac{h_{I}}{c_{L,I}}-1\notin-{\mathbb{Z}}_{>0}$. Then:

\item[(1)] $\Psi|{\mathcal{W}}:{\mathcal{W}}\rightarrow{\mathcal{I}}$ is a linear isomorphism of graded vector spaces.

\item[(2)] $\Psi$ is an isomorphism of $W(2,2)$--modules.
\end{lemma}

\begin{proof}
First we notice that on $\mathcal{I}$
\begin{equation}
W(-n)\equiv2c_{L,I}\left(  \frac{h_{I}}{c_{L,I}}-1+n\right)  I(-n)+\sum
_{i=1}^{n-1} I(-i)I(-n+i)\quad .\label{psi}
\end{equation}
Then
$$
\Psi(W(-n_{1})\dots W(-n_{r})v_{h,h_{W}})=\nu I(-n_{1})\cdots I(-n_{r}%
)v_{h,h_{I}}+\cdots\quad(\nu\neq0)
$$
where $"\cdots"$ denotes the sum of vectors of the form
$$
a\ I(-j_{1})\cdots I(-j_{s})v_{h,h_{I}},\quad j_{1}+\cdots+j_{s}=n_{1}%
+\cdots+n_{r}\quad(a\in\mathbb{C)}%
$$
where $s>r$. This easily implies that the set
$$
\{\Psi(W(-n_{1})\dots W(-n_{r})v_{h,h_{W}})\ |\ n_{1}\geq n_{2}\geq\cdots
n_{r}\geq1;r\in{\mathbb{Z}_{>0}}\}
$$
is linearly independent. This proves assertion (1). The assertion (2)   follows
from (1).
\end{proof}

\begin{example}
$u_{2}=(W(-2)+\frac{6}{c_{W}}W(-1)^{2})v_{h,h_{W}}$ is a singular vector in the Verma module $V^{W(2,2)}(c,c_{W},h,-\frac{c_{W}}{8}).$ Taking $p=2$ in (\ref{psi}), one gets $\Psi(W(-2))=8c_{L,I}I(-2)+I(-1)^{2}$ and
$\Psi(W(-1))=6c_{L,I}I(-1),$ so $\Psi(u_{2})=8c_{L,I}(I(-2)-\frac
{1}{c_{L,I}}I(-1)^{2})v_{h,h_{I}}$ which is a singular vector in the Verma module $V^{\HVir}(c_{L},0,c_{L,I},h,3c_{L,I}).$
\end{example}

\begin{remark}
In the case $\frac{h_{I}}{c_{L,I}}-1=-p\in\mathbb{Z}_{>0}$ $\ker\Psi$ is not trivial. For example, if $h_{I}=0,$ one gets from (\ref{psi}) that
$\Psi\left(  W(-1)\right)  =0.$ Note, that $W(-1)v_{h,0}$ is a singular
vector in $V^{W(2,2)}(c,c_{W},h,0).$ Since Verma modules $V^{W(2,2)}(c,c_{W},h,h_{W})$ and $V^{\HVir}(c_{L},0,c_{L,I},h,h_{I})$ have equal characters, it follows that $V^{\HVir}(c_{L},0,c_{L,I},h,h_{I})$ is not a highest weight $W(2,2)$-module (but contains $U(W(2,2))v_{h,h_{I}}$ as a proper submodule). The structure of $V^{\HVir}(c_{L},0,c_{L,I},h,h_{I})$ as a $W(2,2)$--module is very interesting and shall be investigated in our future publications.
\end{remark}

\begin{theorem}
\label{singular-w22} Assume that $\frac{h_{W}}{c_{W}}=\frac{1-p^{2}}{24}$ and
$p\in\Z_{>0}$. Let $h_{I},c_{L,I}\in{\mathbb{C}}$ such that
$$
c_{W}=-24c_{L,I} ^2,\ h_{W}=h_{I}(h_{I}-2c_{L,I}),\ \frac{h_{I}}{c_{L,I}}-1=p.
$$
Then
$$
\Psi^{-1}\left(  S_{p}(-\frac{I(-1)}{c_{L,I}},-\frac{I(-2)}{c_{L,I}}%
,\cdots)v_{h,h_{I}}\right)
$$
is a singular vector of conformal weight $p$ in the Verma module $V^{W(2,2)}(c_{L},c_{W},\allowbreak h,h_{W})$.
\end{theorem}

\begin{proof}
We have
$$
\frac{h_{W}}{c_{W}}=\frac{h_{I}(h_{I}-2c_{L,I})}{-24c_{L,I} ^2}=\frac{1-p^{2}%
}{24}\quad\mbox{iff}\quad|\frac{h_{I}}{c_{L,I}}-1|=p.
$$
We can choose $h_{I},c_{L,I}$ such that $\frac{h_{I}}{c_{L,I}}-1=p$. Then
$\Psi|\mathcal{W}$ is a linear isomorphism and there is a non-trivial vector
$v\in\mathcal{W}\subset V^{W(2,2)}(c_{L},c_{W},h,h_{W})$ such that
$\Psi(v)=S_{p}(-\frac{I(-1)}{c_{L,I}},-\frac{I(-2)}{c_{L,I}},\cdots
)v_{h,h_{I}}$. The proof follows from the fact that $S_{p}(-\frac
{I(-1)}{c_{L,I}},-\frac{I(-2)}{c_{L,I}},\cdots)v_{h,h_{I}}$ is also a highest
weight vector for $W(2,2)$ and that $L(n){\mathcal{W}}\subset{\mathcal{W}}$
for $n\geq1$.
\end{proof}

\begin{remark}
A family of singular vectors in the Verma modules over $W(2,2)$-algebras were constructed in \cite{R1}. Now we have found a method for constructing all singular vectors which belong to the subspace $\mathcal{W}$. One should take singular vectors $S_{p}(-\frac{I(-1)}{c_{L,I}},-\frac{I(-2)}{c_{L,I}},\cdots
)v_{h,h_{I}}$. The relation (\ref{psi}) gives a system of equations which shows how one can express $I(-n)$ using elements $W(-n)$ of the $W(2,2)$--algebras.
\end{remark}

\vskip 5mm

\noindent {\bf Acknowledgments.} We would like to thank the referees for their valuable comments.
The authors  are partially supported by the Croatian Science Foundation under the project 2634.

\end{document}